\documentclass[11pt]{amsart}

\usepackage[letterpaper,margin=.6in]{geometry}
\usepackage[colorlinks,
linkcolor=blue,
             citecolor=black!75!red,
             pdfproducer={pdfLaTeX},
             pdfpagemode=None,
             bookmarksopen=true
             bookmarksnumbered=true]{hyperref}
\usepackage[matrix,arrow,ps,color,line,curve,frame,all]{xy}

\usepackage{tikz, tikz-cd}
\usetikzlibrary{matrix,arrows,decorations.pathmorphing, fit}
\tikzset{myboxgroup/.style={draw, densely dotted}} 

\makeatletter
\renewcommand*\env@matrix[1][*\c@MaxMatrixCols c]{%
  \hskip -\arraycolsep
  \let\@ifnextchar\new@ifnextchar
  \array{#1}}
\makeatother

\usepackage{amsmath,amsthm,stmaryrd}
\usepackage{cleveref}

\newtheorem{lemma}{Lemma}[section]
\newtheorem{theorem}[lemma]{Theorem}
\newtheorem*{theorem*}{Theorem}
\newtheorem{proposition}[lemma]{Proposition}

\newtheorem{conjecture}[lemma]{Conjecture}

\theoremstyle{definition}

\newtheorem{definitionnodiamond}[lemma]{Definition}
\newtheorem{examplenodiamond}[lemma]{Example}
\newtheorem{remarknodiamond}[lemma]{Remark}

\newtheorem{notationnodiamond}[lemma]{Notation}

\newenvironment{definition}{\begin{definitionnodiamond}}{\hfill\ensuremath\blacklozenge\end{definitionnodiamond}}

\newenvironment{remark}{\begin{remarknodiamond}}{\hfill\ensuremath\blacklozenge\end{remarknodiamond}}

\makeatletter
\let\xx@thm\@thm
\AtBeginDocument{\let\@thm\xx@thm}
\makeatother

\numberwithin{equation}{section}

\newcounter{stepofproof}

\crefname{section}{Section}{Sections}
\crefformat{section}{#2Section~#1#3}
\Crefformat{section}{#2Section~#1#3}

\crefname{subsection}{}{Subsections}
\crefformat{subsection}{\S#2#1#3}
\Crefformat{subsection}{\S#2#1#3}

\crefname{definition}{Definition}{Definitions}
\crefformat{definition}{#2Definition~#1#3}
\Crefformat{definition}{#2Definition~#1#3}

\crefname{example}{Example}{Examples}
\crefformat{example}{#2Example~#1#3}
\Crefformat{example}{#2Example~#1#3}

\crefname{examplenodiamond}{Example}{Examples}
\crefformat{examplenodiamond}{#2Example~#1#3}
\Crefformat{examplenodiamond}{#2Example~#1#3}

\crefname{remark}{Remark}{Remarks}
\crefformat{remark}{#2Remark~#1#3}
\Crefformat{remark}{#2Remark~#1#3}

\crefname{remarknodiamond}{Remark}{Remarks}
\crefformat{remarknodiamond}{#2Remark~#1#3}
\Crefformat{remarknodiamond}{#2Remark~#1#3}

\crefname{convention}{Convention}{Conventions}
\crefformat{convention}{#2Convention~#1#3}
\Crefformat{convention}{#2Convention~#1#3}

\crefname{conjecture}{Conjecture}{Conjectures}
\crefformat{conjecture}{#2Conjecture~#1#3}
\Crefformat{conjecture}{#2Conjecture~#1#3}

\crefname{lemma}{Lemma}{Lemmas}
\crefformat{lemma}{#2Lemma~#1#3}
\Crefformat{lemma}{#2Lemma~#1#3}

\crefname{proposition}{Proposition}{Propositions}
\crefformat{proposition}{#2Proposition~#1#3}
\Crefformat{proposition}{#2Proposition~#1#3}

\crefname{corollary}{Corollary}{Corollaries}
\crefformat{corollary}{#2Corollary~#1#3}
\Crefformat{corollary}{#2Corollary~#1#3}

\crefname{theorem}{Theorem}{Theorems}
\crefformat{theorem}{#2Theorem~#1#3}
\Crefformat{theorem}{#2Theorem~#1#3}

\crefname{assumption}{Assumption}{Assumptions}
\crefformat{assumption}{#2Assumption~#1#3}
\Crefformat{assumption}{#2Assumption~#1#3}

\crefname{enumi}{}{}
\crefformat{enumi}{(#2#1#3)}
\Crefformat{enumi}{(#2#1#3)}

\crefname{equation}{}{}
\crefformat{equation}{(#2#1#3)}
\Crefformat{equation}{(#2#1#3)}

\crefname{align}{}{}
\crefformat{align}{(#2#1#3)}
\Crefformat{align}{(#2#1#3)}

\crefname{proofstep}{Step}{Steps}
\crefformat{proofstep}{#2Step~#1#3}
\Crefformat{proofstep}{#2Step~#1#3}

\newcommand\arXiv[1]{\href{http://arxiv.org/abs/#1}{\nolinkurl{arXiv:#1}}}
\newcommand\MRnumber[1]{\href{http://www.ams.org/mathscinet-getitem?mr=#1}{\nolinkurl{MR#1}}}
\newcommand\DOI[1]{\href{http://dx.doi.org/#1}{\nolinkurl{DOI:#1}}}
\newcommand\MAILTO[1]{\href{mailto:#1}{\nolinkurl{#1}}}

\usepackage{amsfonts,amssymb}

\newcommand\bA{\mathbb A}

\newcommand\bG{\mathbb G}

\newcommand\bP{\mathbb P}

\newcommand\bX{\mathbb X}

\newcommand\bZ{\mathbb Z}

\def\ve{\varepsilon}

\newcommand\cO{\mathcal O}

\def\Proj{\operatorname{Proj}}

\renewcommand\lim{\varprojlim}

\numberwithin{equation}{section}

\title[Flat families of point schemes]{Flat families of point schemes for connected graded
  algebras}

\author[Alex Chirvasitu]{Alex Chirvasitu}
\author[Ryo Kanda]{Ryo Kanda} 	

\address[Alex Chirvasitu]{Department of Mathematics, University at
  Buffalo, Buffalo, NY 14260-2900, USA.}  \email{achirvas@buffalo.edu}

\address[Ryo Kanda]{Department of Mathematics, Graduate School of
  Science, Osaka University, Toyonaka, Osaka, 560-0043, Japan, and
  Department of Mathematics, Box 354350, University of Washington,
  Seattle, WA 98195, USA.}  \email{ryo.kanda.math@gmail.com}

\keywords{non-commutative algebraic geometry}

\subjclass[2010]{14A22, 16S38, 16W50} 

\begin{document}

\pagenumbering{arabic}

\begin{abstract}
  We study truncated point schemes of connected graded algebras as
  families over the parameter space of varying relations for the
  algebras, proving that the families are flat over the open dense
  locus where the point schemes achieve the expected (i.e. minimal)
  dimension.

  When the truncated point scheme is zero-dimensional we obtain its
  number of points counted with multiplicity via a Chow ring
  computation. This latter application in particular confirms a
  conjecture of Brazfield to the effect that a generic two-generator,
  two-relator 4-dimensional Artin-Schelter regular algebra has seventeen truncated point
  modules of length six.
\end{abstract}
                             
\maketitle

\setcounter{tocdepth}{1} 

\section*{Introduction}\label{se.intro}

The context for the present note is that of non-commutative projective
algebraic geometry, in the sense of studying graded algebras and
modules as (analogues of) homogeneous coordinate rings, as
exemplified, for instance, by the seminal paper \cite{AS87}. The
follow-up work of \cite{ATV1,ATV2} introduced novel methods of
handling the difficulties inherent in working with non-commutative
rings by leveraging classical (as opposed to non-commutative)
algebraic geometry to probe the nature of the ``non-commutative projective
schemes'' embodied by the rings in question. We recall the relevant
setup briefly.

To fix ideas and notation, consider an algebraically closed field
$\Bbbk$, an $r$-dimensional vector space $V$ whose dual is spanned by
basis elements $x_i$, $1\le i\le r$, and $s$ multilinear (of degree at least two) forms $f_j$ on $V$. The typical algebra we consider is then
of the form
\begin{equation*}
  A=\frac{T(V^{*})}{I}=\frac{\Bbbk\langle x_{1},\ldots,x_{r}\rangle}{(f_{1},\ldots,f_{s})}
\end{equation*}
(regarding the degree-one generators as elements of the dual $V^*$ is
simply a matter of convention).

A \emph{point module} of $A$ is a graded $A$-module that is cyclic and
has Hilbert series $(1-t)^{-1}$. If $A$ is commutative these
correspond to the closed points of the projective scheme $\Proj A$,
justifying the nomenclature. One of the innovations of
\cite{ATV1} was to introduce a scheme $\Gamma$ whose closed points
parametrize the isomorphism classes of point modules over $A$; this is
the so-called {\it point scheme} of $A$. The scheme $\Gamma$ is the
inverse limit of the \emph{truncated point schemes}
$\{\Gamma_{n}\}_{n}$ defined as follows:

Regard the relations $f_j$ with degrees $d_{j}$ as elements of the
respective tensor powers $(V^*)^{\otimes d_j}$. For every $n\ge 2$ we
define
\begin{equation*}
  \Gamma_n\subset\bP(V)^n\cong\left(\bP^{r-1}\right)^n 
\end{equation*}
to be the zero scheme of the degree-$n$ component $I_n$ of the ideal
$I$ generated by $f_j$s.

The closed points of $\Gamma_{n}$ parametrize the isomorphism classes
of \emph{truncated point modules} of length $n+1$, defined as cyclic
graded $A$-modules with Hilbert series $1+t+t^{2}+\cdots+t^{n}$. If
the number $n$ is larger than or equal to the highest degree of the
defining relations $f_{1},\ldots,f_{s}$, then $\Gamma_{n}$ determines
all truncated point schemes with indices larger than $n$ and hence the
point scheme $\Gamma$.

In the present note we study the behavior of the truncated point
schemes $\Gamma_n$ upon varying the relation space
\begin{equation}\label{eq:rels}
  \mathrm{span}\{f_j\}
\end{equation}
while keeping the degrees $d_j$ of the $f_j$ fixed. In other words, we
regard \Cref{eq:rels} as a point in the relevant product
\begin{equation}\label{eq:g}
  \bG=\prod_{j=1}^s \bP((V^*)^{\otimes d_j})
\end{equation}
of projective spaces, and study $\Gamma_n$ as fibers of a family over
the latter scheme. 

\Cref{th.dns} below shows that under appropriate bounds on the degrees
$d_j$, the locus $U\subset\bG$ over which $\Gamma_n$ has the expected
minimal dimension is open and dense in
$\bG$. Moreover, according to \Cref{th.flt} the resulting family is
flat over $U$. This implies that a suite of algebraic-geometric
invariants one might compute for $\Gamma_n$ (e.g. the arithmetic
genus) stay constant so long as the dimension of $\Gamma_n$ is that
provided by the naive count.

When $\Gamma_{n}$ is zero-dimensional a simple computation in the Chow
ring of $\bP(V)^n$ returns the number of points of $\Gamma_{n}$
counted with multiplicity.  We apply this procedure and our main
results to algebras of the same ``shape'' (i.e. having the same number of generators
and degrees of relations) as the four-dimensional Artin-Schelter
regular algebras listed in \cite[Proposition 1.4]{LPWZ}. There are
three types of such algebras, and in each case we compute the number
of points (counted with multiplicity) of $\Gamma_{n}$ for the smallest number
$n$ such that $\dim(\Gamma_{n})=0$. This includes via \Cref{pr.hlb} a
confirmation in \Cref{pr.12221} of \cite[Conjecture IV.8.1]{brz}:

\begin{conjecture}\label{cj.brz}
  Let $A$ be a connected graded algebra with two degree-one
  generators. If the defining ideal of $A$ is generated by a generic
  cubic and a generic quartic relation, then $\Gamma_5$ consists of
  exactly seventeen distinct points.
\end{conjecture}

\subsection*{Acknowledgements}
We would like to thank Michaela Vancliff for bringing to our attention
the thesis \cite{brz}, Tom Cassidy for making a copy available to us,
and S. Paul Smith for his suggestions for improving an initial draft
of this paper and his insightful remarks on the topics touched on
here.

A. C. is grateful for partial support through NSF grant DMS-1565226.

R. K. was a JSPS Overseas Research Fellow and supported by JSPS
KAKENHI Grant Numbers JP17K14164 and JP16H06337.

\section{Main results}\label{se.gm}

Fix positive integers $r$ and $s$. We consider the following family of algebras associated to an $s$-tuple:

\begin{definition}\label{def.type}
  For a tuple
  \begin{equation*}
    {\bf d} = (d_1\le \cdots\le d_s)
  \end{equation*}
  with $d_j\ge 2$ an {\it algebra of type $(r,{\bf d})$} is a
  connected graded algebra with $r$ degree-one generators and $s$
  relations of degrees $d_{1},\ldots,d_{s}$.
\end{definition}

We retain the notations in the introduction, and focus on $\Gamma_n$ for $n\ge d_s$, henceforth referred to as
the {\it stable range} for $n$. Note that for stable $n$ the scheme
$\Gamma_n$ is defined as the joint zero locus in $\bP(V)^n$ of
\begin{equation}\label{eq:nr-eqns}
  \sum_{j=1}^s (n-d_j+1)
\end{equation}
multilinear equations whose respective degrees are indicated by the
summands of \Cref{eq:nr-eqns}: $n-d_j+1$ equations of degree $d_j$.

\begin{definition}\label{def.def}
  Given $r$, ${\bf d}$ and $n$ as above, the {\it defect}
  $\mathrm{df}(r,{\bf d},n)$ attached to this data is the sum
  \Cref{eq:nr-eqns}.
\end{definition}

Recall that we are denoting by $\bG$ the space \Cref{eq:g} of
relations for type-$(r,{\bf d})$ algebras. It is, in other words, the
variety of tuples of homogeneous polynomials $f_j$ of prescribed
degrees $d_j$ up to scaling. Now consider the closed subscheme
$\bX_{n}$ of $\bG\times\bP(V)^{n}$ defined by
\begin{equation*}
	\{\,((f_{j})_{j},(a_{i})_{i})\in\bG\times\bP(V)^{n}\mid\,\text{$f_{j}(a_{i+1},\ldots,a_{i+d_{j}})=0$ for $1\leq j\leq s$ and $0\leq i\leq n-d_{j}$}\}.
\end{equation*}
The fiber $(\bX_n)_R$ at $R\in \bG$ along the projection $\pi\colon\bX_n\to\bG$
is the scheme $\Gamma_n$ attached to the
type-$(r,{\bf d})$ algebra $TV^*/(R)$.

Our main results are as follows. First, we have the following
observation to the effect that $\Gamma_n$ has ``expected dimension''
generically.

\begin{theorem}\label{th.dns}
  Fix $r$, ${\bf d}$ and $n$, and suppose the associated defect
  $\mathrm{df}$ is $\le n(r-1)$. Then,
  \begin{enumerate}
  \item For each $R\in\bG$, $\Gamma_n=(\bX_n)_R$ is non-empty, and all
    components have dimension $\ge n(r-1)-\mathrm{df}$;
  \item The locus $U$ of $R\in \bG$ where all components of $\Gamma_{n}$ have dimension $n(r-1)-\mathrm{df}$ is open and dense.
  \end{enumerate}
\end{theorem}
\begin{proof}
  We prove the two claims separately.

  \vspace{.5cm}

  {\bf (1)} As observed above, $\Gamma_n$ is by definition the
  scheme-theoretic intersection of $\mathrm{df}$ hypersurfaces in the
  $n(r-1)$-dimensional scheme $\bP(V)^n$, so the lower bound
  $n(r-1)-\mathrm{df}$ for the dimensions of the components is a
  consequence, for instance, of \cite[Proposition I.7.1]{Hart}: that
  result is stated for subschemes of affine space, but our scheme
  $\bP(V)^n$ admits a cover by open patches isomorphic to $\bA^n$.

  The non-emptiness follows from the following computation in the Chow
  ring $A^*=A^*(\bP(V)^n)$. According to the K\"unneth theorem for
  Chow rings (e.g. \cite[Propositions 1 and 2]{tot-ch}) $A^*$ is
  isomorphic to the $n^{th}$ tensor power of $A^*(\bP(V))$, which is
  simply $\bZ[\ve]/(\ve^r)$ for the class $\ve$ of a hyperplane:
  \begin{equation*}
    A^*\cong \bigotimes_{i=1}^n \bZ[\ve_i]/(\ve_i^r). 
  \end{equation*}
  Now consider the multilinearizations $f_{i,j}$ of $f_j$ on
  $\bP(V)^{ n}$ with $0\le i\le n-d_j$. $\Gamma_n$ is the
  intersection of the respective zero loci $V_{i,j}$ of $f_{i,j}$,
  represented in the Chow ring by sums of the form
  \begin{equation}\label{eq:clss}
    \ve_{i+1}+\ve_{i+2}+\cdots+\ve_{i+d_j}.
  \end{equation}
  The product of the elements \Cref{eq:clss} in the Chow ring will be
  shown to be non-zero in \cref{le.prod.nonvan} below. In turn, this
  then implies that the intersection of the schemes $V_{i,j}$
  represented by \Cref{eq:clss} is non-empty.

  To verify this last point, recall e.g. from
  \cite[$\S$8.1]{Fulton-2nd-ed-98} that the product
  \begin{equation*}
    \prod_{j=1}^s \prod_{i=0}^{n-d_j}[V_{i,j}] 
  \end{equation*}
  can be obtained as the pushforward through
  \begin{equation*}
    \bigcap_{j=1}^s \bigcap_{i=0}^{n-d_j} V_{i,j}\to \bP(V)^{ n}
  \end{equation*}
  of an element in the Chow ring of the left hand intersection (see
  especially \cite[Example 8.1.9]{Fulton-2nd-ed-98}). If this
  intersection were trivial then the element in question would vanish,
  hence the conclusion.
  
  \vspace{.5cm}

  {\bf (2)} Since $\bG$ is irreducible, it will be sufficient to prove
  that $U$ is open and non-empty. We relegate non-emptiness to
  \Cref{le.opn} below, and assuming it, we focus here on proving
  openness.
  
  We denote $\bX=\bX_n$ for simplicity and write $\bX_i$ for the
  irreducible components of $\bX$. We can then apply \cite[Exercise
  II.3.22]{Hart} to each restriction and corestriction
  \begin{equation*}
    \pi_i:(\bX_i)_{\mathrm{red}}\to \pi(\bX_i)_{\mathrm{red}}. 
  \end{equation*}
  Part (d) of said exercise shows that the set of points of $\bX_i$
  lying in a fiber of $\pi_i$ of dimension $\ge n(r-1)-\mathrm{df}+1$
  is closed, and hence its image $F_i$ through the projective morphism
  $\pi_i$ is also closed. Since $U$ is the complement of the closed
  finite union $\bigcup F_i$, it is open.
\end{proof}

\begin{lemma}\label{le.prod.nonvan}
	In the context of \cref{th.dns},
	\begin{equation*}
		\prod_{j=1}^{s}\left(\prod_{i=0}^{n-d_{j}}(\ve_{i+1}+\cdots+\ve_{i+d_{j}})\right)
	\end{equation*}
	is a nonzero element of the ring $\bigotimes_{i=1}^{n}\bZ[\ve_{i}]/(\ve_{i}^{r})$.
\end{lemma}
\begin{proof}
  We will prove the statement for all
  $\mathbf{d}=(d_{1},\ldots,d_{s})$ and $n$ with the milder restriction
  $1\leq d_{j}\leq n$ for all $j$ and the same assumption
  $\mathrm{df}\leq n(r-1)$. We may assume
  $1\leq d_{1}\leq\cdots\leq d_{s}\leq n$ without loss of generality.
	
  If we append $d_{s+1}=n$ at the end of $\mathbf{d}$, then the defect
  is increased by one and the element in question is multiplied by
  $\ve_{1}+\cdots+\ve_{n}$. By applying this operation as many times
  as necessary, we can assume $\mathrm{df}=n(r-1)$. Then the
  inequality $r-1\leq s$ follows from
	\begin{equation*}
		n(r-1)=\sum_{j=1}^{s}(n-d_{j}+1)\leq\sum_{j=1}^{s}n=ns.
	\end{equation*}
	
	The number of $j$ with $d_{j}=1$ is at most $r-1$. Indeed, if
        $1=d_{1}=\cdots=d_{r-1}$, then
	\begin{equation*}
		n(r-1)=\sum_{j=1}^{s}(n-d_{j}+1)=n(r-1)+\sum_{j=r}^{s}(n-d_{j}+1)
	\end{equation*}
	and $n-d_{j}+1\geq 1$. Hence $s=r-1$ in this case.
	
	Now we complete the proof by induction on $n$. If $n=1$, then
        $d_{1}=\cdots=d_{s}=1$ and hence $s=r-1$. The element in
        question is $\ve^{r-1}_{1}\neq 0$.
	
	Let $n\geq 2$. For two elements
        $P_{1},P_{2}\in\bigotimes_{i=1}^{n}\bZ[\ve_{i}]/(\ve_{i}^{r})=\bZ[\ve_{1},\ldots,\ve_{n}]/(\ve_{1}^{r},\ldots,\ve_{n}^{r})$,
        we write $P_{1}\leq P_{2}$ if $P_{2}-P_{1}$ is represented by
        a polynomial whose coefficients are all nonnegative. Then we
        have
	\begin{align*}
          \prod_{j\leq r-1}\left(\prod_{i=0}^{n-d_{j}}(\ve_{i+1}+\cdots+\ve_{i+d_{j}})\right)
          &\geq\prod_{j\leq r-1}\left(\ve_{1}\prod_{i=1}^{n-d_{j}}(\ve_{i+1}+\cdots+\ve_{i+d_{j}})\right),\quad\text{and}\\
          \prod_{j\geq r}\left(\prod_{i=0}^{n-d_{j}}(\ve_{i+1}+\cdots+\ve_{i+d_{j}})\right)
          &\geq\prod_{j\geq r}\left(\prod_{i=0}^{n-d_{j}}(\ve_{i+2}+\cdots+\ve_{i+d_{j}})\right).
	\end{align*}
	Therefore
	\begin{equation}\label{eq.prod.nonvan}
		\prod_{j=1}^{s}\left(\prod_{i=0}^{n-d_{j}}(\ve_{i+1}+\cdots+\ve_{i+d_{j}})\right)
		\geq\ve_{1}^{r-1}\prod_{j=1}^{s}\left(\prod_{i=0}^{(n-1)-d'_{j}}(\ve_{i+2}+\cdots+\ve_{i+1+d'_{j}})\right)
	\end{equation}
	where $d'_{j}=d_{j}$ for $j\leq r-1$ and $d'_{j}=d_{j}-1$ for
        $j\geq r$. The right-hand side of \cref{eq.prod.nonvan} is of
        the form $\ve_{1}^{r-1}P$, where $P$ is the element in
        question for the tuple $(d'_{1},\ldots,d'_{s})$ in variables
        $\ve_{2},\ldots,\ve_{n}$. Since the defect for this new tuple
        is
	\begin{equation*}
		\sum_{j=1}^{s}((n-1)-d'_{j}+1)=\sum_{j=1}^{s}(n-d_{j}+1)-(r-1)=(n-1)(r-1),
	\end{equation*}
	the induction hypothesis implies that $P$ is a nonzero element
        of
        $\bZ[\ve_{2},\ldots,\ve_{n}]/(\ve_{2}^{r},\ldots,\ve_{n}^{r})$. Therefore
        both sides of \cref{eq.prod.nonvan} are $\geq 0$ and
        nonzero. This completes the proof.
\end{proof}

\begin{lemma}\label{le.opn}
  In the context of \Cref{th.dns} there are choices of relations
  $f_j$, $1\le j\le s$ for which all components of $\Gamma_n$ achieve
  the lower dimension bound of $n(r-1)-\mathrm{df}$.
\end{lemma}
\begin{proof}
  Simply select the forms $f_j$ to be of the form
  \begin{equation*}
    f_j=\prod_{i=1}^{d_j} \ell_{i,j}
  \end{equation*}
  for linear forms $\ell_{i,j}$ on $\bP(V)$, chosen so that the zero
  locus of any $r$ is empty (i.e. the zero loci
  $Z(\ell_{i,j})$ are in general position in $\bP(V)$).

  The components of the joint zero locus of the multilinearizations of
  the $f_j$ are obtained by imposing $\mathrm{df}$ linear constraints
  on the $n$ components of points in $\bP(V)^{ n}$, and the fact
  that such components have the requisite dimension
  $n(r-1)-\mathrm{df}$ follows from the generic choice of
  $\ell_{i,j}$.
\end{proof}

Additionally, the following result ensures that when $\Gamma_n$ has
the expected size, various invariants such as multi-degrees as
subschemes of products of projective spaces, genus, etc. remain
constant. The result is analogous to \cite[Theorem 4.4]{cs-chl}, and
its proof is similarly based on \cite[Theorem 18.16]{Ebud95}.

\begin{theorem}\label{th.flt}
  The restriction of the family $\bX_n\to \bG$ to the open dense
  subscheme $U\subseteq \bG$ from \Cref{th.dns} is flat.
\end{theorem}
\begin{proof}
  Once more denote $\bX_n$ by $\bX$; we indicate restriction of
  families by subscripts, as in $\bX_U$ for the restriction of
  $\pi:\bX\to \bG$ to $U\subseteq \bG$. We will apply \cite[Theorem
  18.16 (b)]{Ebud95} to the following setup.

  Let $x\in \bX_U$. Then, the theorem in question applies to the local
  rings
  \begin{equation*}
    (R,P) = (\cO_{U,\pi(x)},\mathfrak{m}_{\pi(x)})\to (\cO_{\bX,x},\mathfrak{m}_x) = (A,Q).
  \end{equation*}
  For this, we need
  \begin{itemize}
  \item $R$ to be regular, which it is, being the local ring of a point
    on a product of projective spaces.
  \item $A$ to be Cohen-Macaulay. This follows from the fact that
    $\bX_U\to U$ is a complete intersection in the family
    \begin{equation*}
      \prod_{j=1}^s \bP((V^*)^{\otimes d_j}) \times U\to U.
    \end{equation*}
  \item The dimension of the fiber $A/PA$ equals the relative
    dimension $\mathrm{dim}(A)-\mathrm{dim}(R)$; this is simply a
    paraphrase of the fact that we are restricting to the locus $U$
    where $\pi$ has fibers of the lowest possible expected dimension,
    i.e. $n(r-1)-\mathrm{df}(r,{\bf d},n)$.
  \end{itemize}
  This completes the proof.
\end{proof}

Finally, we end with the following remark that will come in handy
below, when we examine some examples.

\begin{proposition}\label{pr.hlb}
  In the setting of \Cref{th.flt}, suppose furthermore that
  $n(r-1)=\mathrm{df}$. Then, the set $W\subseteq U$ over which
  $\Gamma_n$ is reduced is open and dense.
\end{proposition}
\begin{proof}
  The irreducibility of $U$ means that it is sufficient to prove the
  set in question open and non-empty.
  
  Under the present hypotheses, at each point in $U$ the scheme
  $\Gamma_n$ is finite, i.e. consists of several points, some,
  perhaps, with multiplicity. The flatness result in \Cref{th.flt}
  ensures that the length $|\Gamma_n|$ is constant throughout $U$,
  counting multiplicity; we denote this common number by $\ell$.

  By the functorial description of the Hilbert scheme of points
  (e.g. \cite[p. 15]{Hart66}, \cite[Definition 2.1]{brtn-hlb} or
  \cite[Tag 0B94]{stacks-project}), the flat family $\bX_U\to U$
  entails a map $\phi:U\to \mathrm{Hilb}^\ell_{\bP(V)^n}$.

  The Hilbert scheme contains an open subscheme $\mathrm{Hilb}^\circ$
  consisting of $\ell$-tuples of {\it distinct} points in $\bP(V)^n$
  (see \cite[Proposition 2.4]{brtn-hlb} and the remarks following
  it). In conclusion, the openness and non-emptiness of $W$ will
  follow once we argue that $\phi(U)$ intersects $\mathrm{Hilb}^\circ$,
  i.e. $\Gamma_n$ is reduced for at least one point of $U$.

  To verify this last claim, note that $\Gamma_n$ will indeed be
  reduced for a generic choice of linear forms in the construction
  used in the proof of \Cref{le.opn}.
\end{proof}

\section{Examples and connections to prior work}\label{se.old}

The preceding material ties in with a number of results of similar
flavor in the literature, as we now document.

We will be focusing on algebras with the same generator-relation
pattern as the four-dimensional AS-regular ones classified in
\cite[Proposition 1.4]{LPWZ}:

\begin{itemize}
\item four generators and six relations of degree $2$;
\item three generators, two degree-two relations and two degree-three
  relations;
\item two generators and one relation in each degree $3$ and $4$.
\end{itemize}

Under the regularity assumptions of \cite{LPWZ} the Betti numbers of
these types of algebras are, respectively

\begin{itemize}
\item $1,4,6,4,1$;
\item $1,3,4,3,1$;
\item $1,2,2,2,1$. 
\end{itemize}

When applying the contents of \Cref{se.gm}, the relevant critical
dimension $n(r-1)-\mathrm{df}(r,{\bf d},n)$ becomes zero for certain $n$,
i.e. the $\Gamma_n$ in question will be non-empty finite (perhaps
reduced) schemes.

\subsection{$4$ generators, $6$ quadratic relations (type
  14641)}\label{subse.14641}

In this case, the results of \Cref{se.gm} essentially recapture the
main result of \cite{VdB88} to the effect that generically, such
algebras have $20$ point modules, counted with multiplicity.

In the absence of regularity conditions the scheme $\Gamma_2$ will be
our stand-in for the scheme of point modules, and hence the $n$ to
which \Cref{se.gm} applies here is $2$.

We thus have $r=4$ and $s=6$, and the vector space $V$ of the above
discussion is dual to the span $V^*$ of linearly independent
generators $x_1$ up to $x_4$. The scheme $\bG$ is
$\bP(V^*\otimes V^*)^6$, all $d_j$ are equal to $2$, and the defect is
$6$.

We then have

\begin{proposition}\label{pr.14641}
  Under the conventions of the present subsection, the scheme
  $\Gamma_2$ is non-empty, and the locus $U\subset \bG$ where
  $\Gamma_2$ is zero-dimensional is open and dense.

  For relation spaces $R\in U$, $\Gamma_2$ consists of twenty points,
  counted with multiplicity. These points are distinct for $R\in W$ as
  in \Cref{pr.hlb}.
\end{proposition}
\begin{proof}
  Everything but the claim about the count of $20$ is an immediate
  application of \Cref{th.dns,th.flt,pr.hlb}.

  As for the count itself, it follows from the fact that examples with
  $|\Gamma_2|=20$ exist, as first constructed in \cite{VdB88} (see
  also \cite{SV07,ChV,CS15-2} and references therein) together with
  flatness; the latter ensures the constancy of the degree throughout
  the open parameter family $U$.

  Alternatively, one can avoid having to handle any examples at all by
  resorting to a Chow ring-based argument: $A^*(\bP(V)^2)$ is in this
  case isomorphic to
  \begin{equation*}
    \bZ[\ve_1]/(\ve_1^4) \otimes \bZ[\ve_2]/(\ve_2^4)
  \end{equation*}
  and each bilinearization of a relation cuts out a hypersurface
  $V_i$, $1\le i\le 6$ of class $\ve_1+\ve_2$. Since $\ve_i^4=0$, this
  then implies that the product of the Chow classes $[V_i]$ is
  \begin{equation}\label{eq:ch-prod}
    (\ve_1+\ve_2)^6 = 20~ \ve_1^3\ve_2^3.
  \end{equation}
  On the other hand, \cite[Example~8.2.1]{Fulton-2nd-ed-98}
  implies that the product \Cref{eq:ch-prod} is the
  Chow class of the scheme-theoretic intersection $\bigcap V_i$.
  The assumption of the mentioned result is ensured by \cite[Example~8.2.7]{Fulton-2nd-ed-98}
  and the fact that each $V_{i}$ is Cohen-Macaulay.
  Since $\ve_1^3\ve_2^3$ is the Chow class of a point, this means that said
  intersection consists of twenty points with multiplicity.
\end{proof}

\begin{remark}\label{re.chw}
  It is the second proof of $|\Gamma_2|=20$ given above that would
  presumably be more portable and flexible, as it is available even
  when $\Gamma_n$ is not zero-dimensional. We will treat such a case
  in \cref{pr.12221.4}.
\end{remark}

\begin{remark}\label{re.gor}
	The number $n$ such that the critical dimension becomes zero is equal to $\ell-2$, where $\ell$ is the Gorenstein parameter of a four-dimensional AS-regular algebra of the same generator-relation pattern. Indeed, in the proof of \cite[Proposition 1.4]{LPWZ}, it is observed that the AS-regular algebras considered there have Hilbert series $1/p(t)$, where $p(t)$ has zero at $t=1$ with multiplicity $\geq 3$. In our terminology, $p(1)=0$ implies that $s=2r-2$ and $p'(1)=0$ implies that the sum of $d_j$'s is $(r-1)\ell$. Thus the defect is
	\begin{equation*}
		\sum_{j=1}^{s}(n-d_{j}+1)=(2n+2-\ell)(r-1)
	\end{equation*}
	and it is equal to $n(r-1)$ if and only if $n=\ell-2$.
\end{remark}

\subsection{$3$ generators, quadratic and cubic relations (type
  13431)}\label{subse.13431}

We now tackle the second bullet point listed at the beginning of the
present section, corresponding to three-generator algebras with two
quadratic and two cubic relations. We will then be studying $\Gamma_3$
(i.e. here $n=3$).

\begin{proposition}\label{pr.13431}
  Under the conventions of the present subsection, the scheme
  $\Gamma_3$ is non-empty, and the locus $U\subset \bG$ where
  $\Gamma_3$ is zero-dimensional is open and dense.

  For relation spaces $R\in U$, $\Gamma_3$ consists of nineteen
  points, counted with multiplicity. These points are distinct for
  $R\in W$.
\end{proposition}
\begin{proof}
  The proof is entirely parallel to that of \Cref{pr.14641}, only the
  count requiring modification.
  
  This time, the relevant Chow ring is
  \begin{equation*}
    A^*(\bP^2\times\bP^2\times\bP^2)\cong \bigotimes_{i=1}^{3} \bZ[\ve_i]/(\ve_i^3)
  \end{equation*}
  and the class of $\Gamma_3$ is the coefficient of
  $\ve_1^2\ve_2^2\ve_3^2$ in
  \begin{equation*}
    (\ve_1+\ve_2)^2(\ve_2+\ve_3)^2(\ve_1+\ve_2+\ve_3)^2.
  \end{equation*}
  This is easily seen to be $19$ by direct computation.
\end{proof}

\subsection{$2$ generators, cubic and quartic relations (type
  12221)}\label{subse.12221}

This case is what in fact motivated the present note, and corresponds
to the third bullet point the prefatory discussion to the present
section.

This investigation is a follow-up to \cite{cks} (in turn inspired by
\cite{LPWZ}), and was prompted by our learning belatedly of the thesis
\cite{brz}, where some of the algebras of interest here are
studied. Specifically, the following result resolves \cite[Conjecture
IV.8.1]{brz} (i.e. \Cref{cj.brz}) in the affirmative.

\begin{proposition}\label{pr.12221}
  Under the conventions of the present subsection, the scheme
  $\Gamma_5$ is non-empty, and the locus $U\subset \bG$ where
  $\Gamma_5$ is zero-dimensional is open and dense.

  For relation spaces $R\in U$, $\Gamma_5$ consists of seventeen
  points, counted with multiplicity. The points are distinct for
  $R\in W$.
\end{proposition}
\begin{proof}
  Once more, the argument is precisely parallel to those of
  \Cref{pr.14641,pr.13431}, except for inessential numerical
  differences in the last portion of the proof.

  The Chow ring to consider here is
  \begin{equation*}
    A^*((\bP^1)^5) \cong \bigotimes_{i=1}^5 \bZ[\ve_i]/(\ve_i^2),
  \end{equation*}
  and the sought-after degree is the coefficient of $\prod\ve_i$ in
  \begin{equation*}
    (\ve_1+\ve_2+\ve_3)(\ve_2+\ve_3+\ve_4)(\ve_3+\ve_4+\ve_5)(\ve_1+\ve_2+\ve_3+\ve_4)(\ve_2+\ve_3+\ve_4+\ve_5);
  \end{equation*}
  this is indeed $17$.
  
  Alternatively, we can repeat example-based argument at the end of
  the proof of \Cref{pr.14641}: the family $\bX_U\to U$ is flat, and
  we know that its fiber has degree seventeen for at least one point
  in $U$ via the examples in \cite[Chapter V]{brz}. Flatness then
  ensures that the degree is seventeen throughout $U$.
\end{proof}

For the type of algebras of this subsection, we can apply our general
result also to $n=4$. In this case the expected dimension of
$\Gamma_{4}$ is one.

Recall (e.g. \cite[$\S$2.1.1]{clz}) that the {\it multidegree} of a
closed subscheme $Y$ of a product $\bP=(\mathbb{P}^{r-1})^{\times n}$ of
projective spaces consists of the tuple of cardinalities (with
multiplicities)
\begin{equation}\label{eq:mdg}
  \left|Y\cap \bigcap_{i=1}^{\dim Y} H_i\right|
\end{equation}
for generic choices of hypersurfaces in $\bP$ obtained as zeros of a
single linear form on one of the factors $\bP^{r-1}$ of $\bP$.

When $Y$ is a curve the codimension in \Cref{eq:mdg} is one, and hence
we can express the multidegree simply as a sequence of $n$
non-negative integers
\begin{equation*}
  |Y\cap H_i|, 1\le i\le n
\end{equation*}
where
\begin{equation*}
  H_i = (\bP^{r-1})^{\times (i-1)}\times Z(\ell_i) \times (\bP^{r-1})^{\times (n-i)}
\end{equation*}
for generic linear forms $\ell_i$.

\begin{proposition}\label{pr.12221.4}
  Under the conventions of the present subsection, the scheme
  $\Gamma_4$ is non-empty, and the locus $U\subset \bG$ where
  $\Gamma_4$ is one-dimensional is open and dense.

  For relation spaces $R\in U$, $\Gamma_4$ has multidegree
  $(4,3,3,4)$.
\end{proposition}
\begin{proof}
	The proof is similar to that of \cref{pr.14641}, but now we consider the Chow ring
	\begin{equation*}
		A^*((\bP^1)^4) \cong \bigotimes_{i=1}^4 \bZ[\ve_i]/(\ve_i^2),
	\end{equation*}
	and compute the product
	\begin{equation*}
		(\ve_1+\ve_2+\ve_3)(\ve_2+\ve_3+\ve_4)(\ve_1+\ve_2+\ve_3+\ve_4).
	\end{equation*}
	The result is
	\begin{equation*}
		4\ve_{1}\ve_{2}\ve_{3}+3\ve_{1}\ve_{2}\ve_{4}+3\ve_{1}\ve_{3}\ve_{4}+4\ve_{2}\ve_{3}\ve_{4}.
	\end{equation*}
	The same argument as the latter part of the proof of
        \cref{pr.14641} implies that this is the Chow class of
        $\Gamma_{4}$.

        Finally, the last statement follows from the fact that, as
        explained in \cite[$\S$2.1.1]{clz}, the multidegree can be
        read off as the tuple of coefficients of the Chow class.
\end{proof}



\bibliographystyle{plain}

\def\cprime{$'$}

\end{document}